\documentclass[12pt]{article}
\usepackage{amssymb}
\usepackage{amsfonts}
\usepackage{epsfig}
\usepackage{latexsym}
\usepackage{graphicx}
\usepackage{tikz}
\usetikzlibrary{calc,arrows}
\usepackage{scalefnt}
\usepackage{caption}
\usepackage{url}
\usepackage{amsmath}
\usepackage[margin=1in]{geometry}
\usepackage{verbatim}
\usepackage{parskip}
\usepackage{pgfplots}
\usepackage{float}
\usepackage{amsthm}
\usepackage{float}
\restylefloat{table}

\newcommand{\bbm}{\begin{boldmath}}
\newcommand{\ebm}{\end{boldmath}}

\newcommand{\lt}{<}
\newcommand{\gt}{>}
\newtheorem{theo}{Theorem}[section]
\newtheorem{cor}{Corollary}[section]
\newtheorem{lem}{Lemma}[section]
\newtheorem{prop}{Proposition}[section]

\newtheorem{defy}{Definition}[section]

\newtheorem{mex}{Example}[section]

\newtheorem{conjecture}{Conjecture}[section]
\newtheorem{remk}{{\bf Remark:}}[section]
   {\begin{remk}\begin{normalshape}\mbox{}}%
   {\hfill  \end{normalshape}\end{remk}}%

\tikzstyle{level 1}=[sibling distance=40mm]
\tikzstyle{level 2}=[sibling distance=30mm]
\tikzstyle{level 3}=[sibling distance=20mm]
\tikzstyle{level 4}=[sibling distance=10mm]

\begin{document}
\begin{center}
{\Large {\bf The Fundamental Theorem of Calculus with Gossamer numbers}} \\
\vspace{0.5cm}
{\bf Chelton D. Evans and William K. Pattinson} 
\end{center}
\begin{abstract}
 Within the gossamer numbers $*G$ which extend $\mathbb{R}$ to
 include infinitesimals and infinities we prove
 the Fundamental Theorem of Calculus (FTC).
 Riemann sums are also considered in $*G$,
 and their non-uniqueness at infinity.
 We can represent the sum as a continuous function in $*G$
 by inserting infinitesimal intervals at the discontinuities,
 and threading curves between the
 sums discontinuities.
 As the FTC is a difference of integrals at the end points, the same 
 is true for sums.
\end{abstract}
\section{Introduction}\label{S01}
\ref{S01}. Introduction \\
\ref{S02}. Riemann sums \\
\ref{S03}. Fundamental Theorem of Calculus (FTC) \\
\ref{S04}. Heavyside function and integration \\
\ref{S05}. FTC for sums \\

 While the heavyside step function is well known, 
 it can also be represented
 using the Iverson bracket
 notation (see \cite[p.24]{graham})
 where the logical statement evaluates to $1$ if true, else $0$.
 The prime-number function could be expressed as
 $\pi(x) = [x \geq 2] + [x \geq 3] + [x \geq 5] + [x \geq 7] + \ldots$ 
 As the name ``step" suggests, a discontinuity exists between
 the levels. 

 We consider a step function's continuous representation
 with the gossamer numbers \cite{cebp11}.
 The discontinuity in $\mathbb{R}$ (or $*G$) is replaced
 by a continuous function in $*G$ over an infinitesimal interval. 

 A discontinuous curve in $\mathbb{R}$ can become a continuous 
 curve in $*G$.
 Since we have theory for continuous curves such as FTC,
 having a theory that allows the transfer from the discontinuous
 to the continuous and back is useful. 
 A continuous representation is advantageous as we gain the benefits
 of continuity and classical calculus in $*G$.

 The next benefit of working in $*G$ is the extension of
 the interval to include infinitesimals and infinities,
 thereby bypassing the need to extend $\mathbb{R}$,
 and potentially representing the theory in a simpler way.
 For example, theory may encompass both the finite and
 the infinite intervals and not need or reduce to special cases.
 
 The fundamental theorem of
 calculus in $\mathbb{R}$ is
 stated for a finite interval, and extended
 to improper integrals. The Lebesgue
 measure and integral are a generalization,
 and are not limited by the real numbers
 not having infinitely small and large numbers,
 well defined.
 We see $*G$ as possibly another way. 
 Our understanding is that Non-Standard Analysis (NSA) is also
 a two-tiered calculus and also has the benefits of the infinitesimals and infinities.
 Having an extended domain rather than extending the reals is beneficial.

 Consequently, we believe the fundamental theorem
 of calculus in $*G$ is better than otherwise extending
 calculus for specific cases.
\section{Riemann sums}\label{S02}
 The typical way of visualising
 a Riemann sum is to partition a function into bars,
 where as the columns become infinitesimal,
 the sum of the bars evaluates to the area under the
 curve.

 Contrast this method of integration with Archimedes' method of exhaustion
 \cite{archimedes}.
 We give an example where the area is underestimated,
 but at infinity is equal to the area of the circle.
\bigskip
\begin{mex}\label{MEX002}
 Consider a function of straight lines
 through an ordered list of points on a circle about the origin
 ($0$ to $2\pi$). 
 At infinity, having
 the points cover the circumference,
 the function length equals the circle's circumference.

 An iteration of the sum follows.
  A point is inserted into the ordered
 sequence of points about origin,
 and the function is subsequently updated. 

 With the insertion of each additional point,
 a construction of two triangles' area
 is added to a cumulative sum.
 In this way the circle's area is being
 integrated. 
 Starting with a circle and an inscribed square,
 points are added indefinitely and the area of
 the circle is calculated.
\end{mex}
\bigskip
\begin{mex}\label{MEX003}
A Riemann sum of the above,
 for any line segment of the function forms a triangle
 with the origin.
 
 From a numerical point of view, in the Riemann sum
 the thin isosceles triangles will have 
 roundoff errors
 associated with the sum. However, if the sum can be evaluated symbolically
 then this may not be an issue.
 [ In contrast, the method of exhaustion maintains the triangles shape, even
 for the infinitely small ] 
\end{mex}

 The above are examples of definite integrals, 
 of which the Riemann sum was investigated as a theory of the sum of integrals.
 We consider without loss of generality, positive Riemann sums in $*G$.
\bigskip
\begin{defy}\label{DEF005}
 A uniform Riemann sum in $*G$; $\nu \in \mathbb{N}_{\infty}$;
\[ \sum_{j=1}^{\nu} f(\frac{j}{\nu}) \frac{1}{\nu}|_{\nu=\infty} \]
\end{defy}

 As a uniform partition, with arbitrary infinitely small partitions,
 can divide any
 interval with an infinitely small width.
 Since there is no smallest number,
 and there is no
 smallest interval,
 then we can change the interval width.
\bigskip
\begin{mex}\label{MEX004}
 $\sum_{j=1}^{n} f(j)\frac{1}{n}|_{n=\infty}$
 has $1 \ldots n-1$ steps and width $\frac{1}{n}|_{n=\infty}$.
 $\sum_{j=1}^{n^{2}}f(\frac{j}{n^{2}})\frac{1}{n^{2}}|_{n=\infty}$
 has a $1 \ldots n^{2}-1$ steps and  width $\frac{1}{n^{2}}|_{n=\infty}$.
 $\sum_{j=1}^{n!} f(\frac{j}{n!}) \frac{1}{n!}|_{n=\infty}$
 are 
 uniform Riemann
 sums.
\end{mex}

 Another visual
  way to understand a Riemann sum and integral is 
 to consider the Riemann sum \textit{not} over a finite interval $[a,b]$, 
 but as an infinite positive interval in $*G$.
 The columns and area under the graph are asymptotic,
 which we describe with the asymptotic relation $~$.
\bigskip
\begin{defy}\label{DEF004}
 We say a uniform Riemann sum (Definition \ref{DEF005}) is uniform Riemann integrable; $\nu \in \mathbb{N}_{\infty}$:  
\[ \int_{j}^{j+1} f(\frac{x}{\nu}) \,dx \sim f(\frac{j}{\nu}) \]
\[ \int_{0}^{\nu} f(\frac{x}{\nu})\,dx \sim \sum_{j=1}^{\nu} f(\frac{j}{\nu}) |_{\nu=\infty} \]
\end{defy}
\bigskip
\begin{remk}
Consider
 $\int_{0}^{\nu} f(\frac{x}{\nu})\frac{1}{\nu}\,dx \sim \sum_{j=1}^{\nu} f(\frac{j}{\nu}) \frac{1}{\nu} |_{\nu=\infty}$,
 $\int_{0}^{\nu} f(\frac{x}{\nu}) \frac{d(\frac{x}{\nu})}{dx} \,dx \sim \sum_{j=1}^{\nu} f(\frac{j}{\nu}) \frac{1}{\nu} dj |_{\nu=\infty}$
 where $dj$ is a
 change in integers (see \cite{cebp4}), $dj = (j+1)-j=1$; we see the one to one
 correspondence of the Riemann sum to the integral, as the
 columns are asymptotic to the integral between integer values.
 With this we can interpret between
 a continuous change
 and a discrete change
 of variable.
\end{remk}

 These conditions arise from transforming
 the definite integral
 to the uniform Riemann series, and vice versa. 

 However, in the transformation, while the definite integral 
 $\int_{0}^{1}f(x)\,dx$ is used with the chain
 rule, this does not exclude indefinite
 integral evaluation as the following example
 shows.
\bigskip
\begin{mex}\label{MEX005}
 Integrate the divergent integral
 $\int n^{2}\,dn = \frac{1}{3}n^{3}|_{n=\infty}$
 with a Riemann sum.

 We will use the closed form formula for the sum of
 squares, $\sum_{k=1}^{n}k^{2} = \frac{2n^{3}+3n^{2}+n}{6}$,
 for convenience we can find this using a symbolic 
 maths package: in Maxima `$\,\mathrm{sum(k^2,k,1,n),simpsum;}$'.

 $\int_{1}^{n} x^{2}\,dx|_{n=\infty}$
 $=\int_{\frac{1}{n}}^{1} (nx)^{2} d(nx)|_{n=\infty}$
 $=n^{3} \int_{0}^{1} x^{2}\,dx|_{n=\infty}$.
 Evaluate the definite integral with a finite
 uniform Riemann sum,
 $\int_{0}^{1} x^{2}\,dx$
 $= \sum_{k=1}^{\nu} (\frac{k}{\nu})^{2} \frac{1}{\nu}|_{\nu=\infty}$,
 $= \frac{1}{\nu^{3}} \sum_{k=1}^{\nu} k^{2}|_{\nu=\infty}$
 $= \frac{1}{\nu^{3}} \frac{2\nu^{3}+3\nu^{2} + \nu}{6}|_{\nu=\infty}$
 $= \frac{1}{\nu^{3}} \frac{2\nu^{3}}{6}|_{\nu=\infty}$
 $= \frac{1}{3}$, then
 $\int_{1}^{n} x^{2}\,dx|_{n=\infty}$
 $= \frac{1}{3}n^{3}|_{n=\infty}$.
\end{mex}
 
 To understand Riemann sums better,
 we consider some of the limitations
 of the Riemann sum and integral \cite{riemannlimitations}.

 That the usability for an infinite domain is limited.
 This may well be true, as a Riemann sum is a primitive
 sum.
 However, the Riemann sum can be constructed in $*G$,
 extending the domain to include infinitesimals and
 infinities. 
 Example \ref{MEX005} shows that with scalings, a divergent
 integral can be summed.
 Thus, the possibility of
 addressing the limitation
 of the infinite bounds exists.

 Another limitation of the Riemann sum and
 integral such
 as not being able to integrate
 over the enumeration of rational numbers
 in $[0,1]$, does not invalidate
 the sum for other purposes.  

 Concerns regarding the functions being
 summed, are problem or domain dependent.

 If we restrict the Riemann sum to continuous
 functions in $\mathbb{R}$ and $*G$ the scope
 is still large, particularly in light
 of transference \cite[Part 4]{cebp17}. 

 For example, consider
 a Riemann sum to a finite value.
 While we consider the
 Riemann sum native to $*G$  
 (because this includes infinitesimals and infinities),
 as a two-tiered calculus
 we can have a uniform Riemann sum equal
 to a Riemann sum at infinity after transference (see Conjecture \ref{P007}).
\[ \sum_{j} \int f_{j}(x)\,dx \sim \int \sum_{j} f_{j}(x)\,dx \]
 Interchanging limit processes we believe is
 more fundamental. However, we see mathematics
 with infinity, as the necessary way forward.
 Orderings of variables reaching infinity before
 others is a much larger question, but
 one that has to be asked 
 (See \cite[Part 6 A two-tiered calculus]{cebp20}).
 See our use of interchanging point evaluations in
 the proof of Theorem \ref{P017}.

 Assuming continuous functions,
 uniform Riemann sums are preferred, because they are easier to
 calculate with. 
 Hence, while the Riemann sum may appear more
 general, for the infinite partitions which are important,
 they are actually equivalent. That is, their definition
 is no more general for the infinite case than the uniform partition.
\bigskip
\begin{conjecture} \label{P007}
A uniform Riemann sum of a continuous curve
 is 
 asymptotic 
 to 
 a Riemann sum with infinite 
 limit.
\end{conjecture}

 In forming a Riemann sum \cite{riemannhowto} advised 
 to express the sum in a particular way and apply 
 Corollary \ref{P006}.
\bigskip
\begin{prop}\label{P003}
 With the Riemann integral Definition \ref{DEF004},
 there exists $\nu|_{n=\infty} \in \mathbb{N}_{\infty}$;
 $c \in *G$; 
 $c \prec \sum_{j=1}^{\nu} f(\frac{j}{\nu})\frac{1}{\nu}|_{n=\infty}$
 and $c \prec \int_{0}^{1} f(x)\,dx$:
\[  \sum_{j=1}^{\nu} f(\frac{j}{\nu})\frac{1}{\nu}|_{n=\infty} = \int_{0}^{1} f(x)\,dx + c \]
\end{prop}
\begin{proof}
 Follows $a \sim b$ then there exists $c$:
 $a+c = b$,
 $a \succ c$ and $b \succ c$
 \cite[Part 5, Proposition 2.1]{cebp15}.
\end{proof}
\bigskip
\begin{cor}\label{P006} A Riemann sum may be evaluated by the following.
\[ \lim_{n \to \infty} \sum_{j=1}^{n} f(\frac{j}{n}) = \int_{0}^{1} f(x)\,dx \]
\end{cor}

 However, this does not describe the transition between sum and integral,
 although it is the most practical approach to evaluating a Riemann sum.
 We discuss the transition between Riemann sums and integrals in both
 directions, in the context of sum and integral scaling
 and shifting. 
\bigskip
\begin{prop} \label{P014}
 $c, x \in *G$;
 When $c$ is constant, $d(x+c) = dx$ 
\end{prop}
\begin{proof}
  Let $u=c+x$, $\frac{du}{dx}=1$, $d(c+x) = du = \frac{du}{dx}dx = dx$
\end{proof}
\bigskip
\begin{prop}
 When $\alpha$ is constant, $d(\alpha x) = \alpha dx$
\end{prop}
\begin{proof}
Let $v = \alpha x$, $dv = \frac{dv}{dx}dx = \alpha dx$
\end{proof}

 If we apply the inverse operation
 in the integrand to
 within the integral,
 for
 operators $+$ and $\times$
 we can 
 shift and scale the integral.
 This can be shown by substitution.
\bigskip
\begin{theo} Scaling the integral; $a, b, \alpha, \nu, x, f \in *G$;
\[ \int_{a}^{b} f(x)\,dx = \int_{\frac{a}{\alpha}}^{\frac{b}{\alpha}} f(\alpha \nu) \, d(\alpha \nu) \]
\end{theo}
\begin{proof}
 $\int_{a}^{b} f(x)\,dx$.
 Let $x = \alpha \nu$.
 $x=a$ then $\nu = \frac{a}{\alpha}$,
 $x=b$ then $\nu = \frac{b}{\alpha}$.
 $\int_{a}^{b} f(x)\,dx$
 $= \int_{\frac{a}{\alpha}}^{\frac{b}{\alpha}} f(\alpha \nu) \frac{dx}{d \nu}d \nu$
 $= \alpha \int_{\frac{a}{\alpha}}^{\frac{b}{\alpha}} f(\alpha v) d\nu$
 $= \int_{\frac{a}{\alpha}}^{\frac{b}{\alpha}} f(\alpha \nu) d(\alpha \nu)$
\end{proof}
\bigskip
\begin{theo} Shifting the integral; $a, b, \alpha, \nu, x, f \in *G$;
\[ \int_{a}^{b} f(x)\,dx = \int_{a-c}^{b-c}f(x+c)\,dx \]
\end{theo}
\begin{proof}
 $\int_{a}^{b} f(x)\,dx$.
 Let $x = v + c$, 
 $x=a$ then $v= a-c$,
 $x=b$ then $v=b-c$.
 $\int_{a}^{b} f(x)\,dx$
 $=\int_{a-c}^{b-c} f(v+c)\frac{dx}{dv}dv$
 $=\int_{a-c}^{b-c} f(v+c)dv$

The same result could be achieved by shifting the integrand arguments, and
 shifting in the opposite direction the variable in the body. 
 $\int_{a}^{b} f(x)\,dx$
 $= \int_{a-c}^{b-c}f(x+c)\,d(x+c)$
 $= \int_{a-c}^{b-c}f(x+c)\,dx$
\end{proof}
\bigskip
\begin{prop}\label{P008}
 A definite integral of a continuous function in $*G$
  can be transformed to a Riemann sum,
 and vice versa, if
 uniform Riemann integrable.
$\nu \in \mathbb{N}_{\infty}$; $f, x \in *G$;
\begin{align*}
 \int_{0}^{1} f(x)\,dx & \tag{Scale to an improper integral} \\
 = \int_{0}^{\nu} f(\frac{x}{\nu})\,d(\frac{x}{\nu})|_{\nu=\infty} &  \\
 = \int_{0}^{\nu} f(\frac{x}{\nu})\frac{1}{\nu}dx|_{\nu=\infty} & \tag{Proposition \ref{P003}; scaling out the infinitesimal change}  \\
 = \sum_{j=1}^{\nu} f(\frac{j}{\nu})\frac{1}{\nu}dj|_{\nu=\infty}+c & \tag{the magnitude of the sum dominates} \\
 = \sum_{j=1}^{\nu} f(\frac{j}{\nu})\frac{1}{\nu}|_{\nu=\infty} 
\end{align*}
 Reversing the above operations transforms
 the Riemann sum back into the definite integral.
\end{prop}
\section{Fundamental Theorem of Calculus (FTC)}\label{S03}
 The following proofs are a derivative 
 from those given by D. Joyce \cite{joyce},
 but in $*G$ number system.
 Since this number system includes
 infinitesimals an infinities,
 the domain and space are
 significantly increased.
 For example, FTC is proved for finite
 bounds, but in $*G$, the bounds $[a,b]$ 
 can represent
 improper integrals.
\bigskip
\begin{theo}\label{P017}
 $\mathrm{FTC}^{-1}$ the derivative of an integral of a function is the function. 

$f, F, a, b, x, t \in *G$; 
If $f$ is a continuous
 function on the closed interval $[a,b]$,
 and $F$ is its accumulation function
 defined by
\[ F(x) = \int_{a}^{x} f(t)\,dt \]
 for $x \in [a,b]$, then
 $F$ is differentiable on $[a,b]$
 and its derivative is $f$, that is,
 $F'(x) = f(x)$.
\end{theo}
\begin{proof}
 $ h \in *G$; 
 $\,F'(x)$
 $= \frac{ F(x+h)-F(x)}{h}|_{h=0}$
 $= \frac{1}{h} (\int_{a}^{x+h} f(t)\,dt - \int_{a}^{x} f(t)\,dt )|_{h=0}$
 $= \frac{1}{h} \int_{x}^{x+h}f(t)\,dt|_{h=0}$
 $= \frac{1}{h} \int_{0}^{h} f(t+x)\,d(t+x)|_{h=0}$
 $= \frac{1}{h} \int_{0}^{h} f(t+x)\,dt|_{h=0}$
 as $x$ is constant with respect to $t$.

 Use is made of a one variable reaching
 its value before another, hence
 justifying interchanging point evaluations.
 By sending $h \to 0$ before performing
 integration,
 we are investigating infinitely close
 to the point $x$.

 By constructing a Riemann sum,
 integrating infinitesimally close to the point about $x$,
 $h \to 0$ before $dt \to 0$,
 $\int_{0}^{h} f(x+t)\,dt|_{h=0}$
 $= \int_{0}^{1} f(x+ht)\, h dt|_{h=0}$,
 $=h \sum_{k=1}^{n} f(x+h\frac{k}{n})\frac{1}{n}|_{n=\infty}|_{h=0}$
 $=h \sum_{k=1}^{n} f(x+h\frac{k}{n})|_{h=0}\frac{1}{n}|_{n=\infty}|_{h=0}$
 $=h \sum_{k=1}^{n} f(x+0)\frac{1}{n}|_{n=\infty}|_{h=0}$
 $= h f(x)|_{h=0}$.
 Since $h$ is arbitrarily small.

$F'(x)$
 $= \frac{1}{h} h f(x)|_{h=0}$ 
 $=f(x)$
\end{proof}
\bigskip
\begin{remk}
 If the order of variables reaching infinity is interchanged,
 $n \to \infty$ before $h \to 0$,
 we obtain a different sum. This is
 a consequence of non-uniqueness at
 infinity, corresponding with the different possibilities there.

 $h \sum_{k=1}^{n} f(x+h\frac{k}{n})\frac{1}{n}|_{n=\infty}|_{h=0}$
 $=h \sum_{k=1}^{n} f(x+h\frac{k}{n})|_{n=\infty}\frac{1}{n}|_{n=\infty}|_{h=0}$
 $=h \sum_{k=1}^{n} f(x+\frac{k}{n})|_{n=\infty}\frac{1}{n}|_{n=\infty}|_{h=0}$
 $= h \int_{0}^{1} f(x)\,dx|_{h=0}$
\end{remk}
\bigskip
\begin{theo}\label{P018}
 FTC

If $F'$ is continuous
 on $[a,b]$,
 then
\[ \int_{a}^{b} F'(x)\,dx = F(b)-F(a) \]
\end{theo}
\begin{proof}
 $G(x) = \int_{a}^{x} F'(t)\,dt$.
 By $\mathrm{FTC^{-1}}$ $G'(x) = F'(x)$,
 integrate then
 $G(x) = F(x)+c$ where $c \in *G$ is constant.
 But $G(a)=0$, $0 = F(a)+c$,
 $c = -F(a)$.

 $G(x) = F(x)-F(a)$,
 $G(b)  = F(b) - F(a)$,
 $\int_{a}^{b} F'(t)\,dt = F(b)-F(a)$
\end{proof}
\section{Integrating the heavyside step function}\label{S04}
Given that between
 any two real numbers, there
 exists another real number,
 may give the impression that the
 real number line is dense and complete.
 However, this is not the case.

At a finer layer, the gossamer numbers $*G$
 fill the gaps between
 the real numbers,
 with infinitesimals, as well as extending
 number line.
\bigskip
\begin{lem} \label{P009}
If $a_{i} \in \mathbb{R}$, $h \in \Phi$;
 then $\sum_{k=1}^{\infty} a_{i} h^{k} \in \Phi$.
\end{lem}
\begin{proof}
$\exists \beta \in \mathbb{R}:$ $\forall k$, 
 $|a_{k}| \leq \beta$,
  $\sum_{k=1}^{\infty} a_{i} h^{k} \leq \sum_{k=1}^{\infty} \beta h^{k}$
 
$\theta = h + h^{2} + h^{3} + \ldots$,
 $\theta h = h^{2} + h^{3} + h^{4} + \ldots$,
 $\theta(1-h) = h$, $\theta = \frac{h}{1-h} \in \Phi$

  $\sum_{k=1}^{\infty} a_{i} h^{k} \leq \beta \theta \in \Phi$
\end{proof}
\bigskip
\begin{theo}
 Given $x \in \mathbb{R}$ and $y \in *G$: $x \simeq y$,
 there exist and infinity of numbers infinitesimally
 close to $x$ which transfer $*G \mapsto \mathbb{R}$
 to the unique $x$.
\end{theo} 
\begin{proof}
$h, \delta \in \Phi$; $a_{k} \in \mathbb{R}$;
 $\delta = \sum_{k=1}^{\infty} a_{k} h^{k}$,
 $y = x + \delta$,
 taking the standard part
 $\mathrm{st}(y)=x$ as from Lemma \ref{P009}  $\delta \in \Phi$,
 $\mathrm{st}(\delta)=0$.
\end{proof}

The heavyside step function, discontinuous in $\mathbb{R}$ can,
 with infinitesimals,
 be representative of a continuous function in $*G$. 
 There is no contradiction, because during the transfer 
  $*G \mapsto \mathbb{R}$ the continuous function becomes discontinuous as the
 infinitesimals are projected to $0$.

 Let $g$ be the function over the discontinuities in $*G$.
 Let $f_{2} = f + g$. $(*G,f_{2}) \mapsto (\mathbb{R},f)$,
 $(*G,\int \! f_{2}\,dx) \mapsto (\mathbb{R},\int\! f\,dx)$.
 Construct the integral summing the discontinuities,
 which is also
 an infinitesimal
 $\int g\,dx \in \Phi$.

 The transfer from $*G$ to
 $\mathbb{R}$ would discard the infinitesimal sum as $\Phi \mapsto 0$.

\begin{figure}[h]
\centering
\begin{tikzpicture}
  \draw (0,1) to (1.5,1);
  \draw (1.5,1) [dashed] to (1.5,2);
  \draw (1.5,2) to (2.5,2);
  \draw[fill=white] (1.5,2.0) circle(0.07cm);
  \draw[fill=black] (1.5,1.0) circle(0.07cm);
  \node [label={[shift={(1.5cm,0.2cm)}]$q$}] {};
  \node [label={[shift={(1.5cm,2.8cm)}]$f \in \mathbb{R}$}] {};
  \node [label={[shift={(1.5cm,2.0cm)}]$f \not\in C^{0} $}] {};

\begin{scope}[xshift=4cm];
  \draw (0,1) to (1.05,1);
  \draw (1.05,1) to (1.95,2);
  \draw (1.95,2) to (2.5,2);
  \node [label={[shift={(1.5cm,0.2cm)}]$q$}] {};
  \node [label={[shift={(1.5cm,2.0cm)}]$f_{2} \in C^{0}, \;\; f_{2} \not\in C^{1} $}] {};
  \node [label={[shift={(1.5cm,2.8cm)}]$f_{2} \in *G$}] {};
  \draw[arrows=->](1.7,0.2)--(1.95,0.2);
  \draw[arrows=->](1.3,0.2)--(1.05,0.2);
  \node at (1.5,0.2) { \small $2 \epsilon$ };
  \draw[fill=black] (1.95,2.0) circle(0.07cm);
  \draw[fill=black] (1.05,1.0) circle(0.07cm);
\end{scope}
\begin{scope}[xshift=8cm];
  \draw (0.0,1.0) -- (1.0,1.0);
  \draw plot[samples=600,domain=1.0:2.0] function {1.0*1/(1+exp(16.0*(-x+1.5)))+1};
  \draw (2.0,2.0) -- (2.5,2.0);
  \node [label={[shift={(1.5cm,2.0cm)}]$f_{2} \in C^{\infty}$}] {};
  \node [label={[shift={(1.5cm,2.8cm)}]$f_{2} \in *G$}] {};
  \node [label={[shift={(1.5cm,0.2cm)}]$q$}] {};
  \draw[arrows=->](1.7,0.2)--(1.95,0.2);
  \draw[arrows=->](1.3,0.2)--(1.05,0.2);
  \node at (1.5,0.2) { \small $2 \epsilon$ };
  \draw[fill=black] (1.95,2.0) circle(0.07cm);
  \draw[fill=black] (1.05,1.0) circle(0.07cm);
\end{scope}
\end{tikzpicture}
\caption{Discontinuous heavyside step function continuous in $*G$} \label{fig:FIG01}
\end{figure}
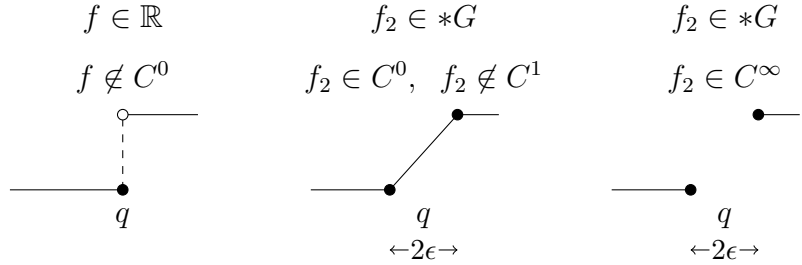

 Let $f$ be continuous, $(*G,f) \mapsto (\mathbb{R},f_{2})$ is not unique.
 The continuous function in $*G$ can
 transfer back to include or not include the discontinuous points.
 The transferred curve may not be a function.

 If we construct a
 straight line in $*G$
 between the discontinuities in $\mathbb{R}$,
 we can show that the $*G$ representation possesses
 the same area when the curve is transferred back. 
 For the heavyside step function, at the discontinuity,
 the infinitesimal triangle area is $\frac{1}{2}(2 \epsilon) \cdot 1 = \epsilon$.

\bigskip
\begin{prop} \label{P011}
For any function with finite or countable discontinuities
 $(\mathbb{R},f): \;\; \exists f_{2}:(*G,f_{2}) \mapsto (\mathbb{R},f)$
\end{prop}
\begin{proof}
 Since $\mathbb{R}$ is embedded in $*G$ this is always true.
\end{proof}
\bigskip
\begin{prop} \label{P012}
For any function $(\mathbb{R},f): \;\; \exists f_{2}:(*G,f_{2}) \mapsto (\mathbb{R},f)$ and $f_{2} \in \mathbb{C}^{w}$
\end{prop}
\begin{proof}
 Between any discontinuous ordered $x \in \mathbb{R}$ values,
 construct a function in $*G$ that is continuous,
 to the desired continuity condition $\mathbb{C}^{w}$.
 For example, at a singularity, construct an s-curve with infinitesimals.
\end{proof}
 
  \textbf{Condition:} Without loss of generality, consider the curves greater than or equal
 to zero and on a positive domain. $x \in *G: 0 \lt x \lt \infty$
\bigskip
\begin{prop}\label{P001}
 For any piecewise continuous function on a positive interval
 with a finite number of discontinuities,
 in $*G$
 a continuous curve can be constructed,
 that transferred back represents the original curve and
 has the same area.
\end{prop}
\begin{proof}
 By Proposition \ref{P012} we can construct a continuous
 curve.
 
 The area under the discontinuities in $*G$ is a trapezoid.
 Let $y_{k}$ be the points of discontinuity,
 let the area at the discontinuity be described
 as a sum of the square and triangle, then 
 let $\epsilon \in \Phi$, 
 $2 \epsilon$ be the interval width, 
 the trapezoid area between the discontinuities 
 $a_{k} = |(y_{k+1}-y_{k})| \epsilon + \mathrm{min}(y_{k+1},y_{k}) 2 \epsilon$. 

 Area of infinitesimal discontinuities
 $\sum_{k=0}^{w-1} a_{k}$
 $= \sum_{k=0}^{w-1} \epsilon (|(y_{k+1}-y_{k})| + 2 \,\mathrm{min}(y_{k+1},y_{k}) )|_{\epsilon=0}$ 
 $= \sum_{k=0}^{w-1} \epsilon |_{\epsilon=0} \in \Phi$, $(*G, \Phi) \mapsto (\mathbb{R},0)$.
\end{proof}
The transfer principle is not only between different number systems,
 but can be applied within $*G$ as it is a `realization' - the truncation
 of terms.
\bigskip
\begin{theo}\label{P002}
 For a given function we can always construct a continuous function in $*G$,
 which can transfer back to the given function.
 If $A$ is the area of the original function,
 and $A'$ is the area of the continuous function, $A' \simeq A$ can be found.
\end{theo}
\begin{proof}
Let $a_{k}(x)$ be a continuous function with an infinitesimal width
 that describes the discontinuities of a function in $*G$.
 Consequently $a_{k}(x)$ has the property that $\sum_{k=0}^{w-1} a_{k}(x) \in \Phi$.

 The condition required to be satisfied is     
 $\sum_{k=0}^{w-1} a_{k}(x) \in \Phi$,
 as the infinitesimal realized becomes an additive identity.
 
 The domain need not be finite, for example an improper Riemann sum.
 The area $A$ could be an infinity.

 With infinitesimals, using the linear approximation between discontinuities,
 the area of infinitesimal discontinuities
 $\sum_{k=0}^{w-1} a_{k}$
 $= \sum_{k=0}^{w-1} \epsilon (|(y_{k+1}-y_{k})| + 2 \mathrm{min}(y_{k+1},y_{k}) )|_{\epsilon=0} \in \Phi$
 since there is no smallest infinitesimal, we can always choose $\epsilon$ to satisfy this condition.
\end{proof}
\bigskip
\begin{cor}\label{P005}
If the given curve is differentiable on one side of the discontinuity
 then 
 a curve can be 
 constructed that is differentiable.
\end{cor}
\begin{proof}
 Have the joining function between the discontinuities be differentiable
 everywhere in its infinitesimal domain and at the joins.  
\end{proof}
\bigskip
\begin{mex}\label{MEX001}
$f(x) = [x \gt q]$ in $\mathbb{R}$ \\

Let $\epsilon \in \Phi$ be an arbitrarily small infinitesimal.
 Let $f(x) \in *G$. 
 $y = \frac{1}{2 \epsilon} x + \frac{1}{2}(\frac{q}{\epsilon}+1)$ \\
$f(x) = [ x \in (q-\epsilon ,q+ \epsilon ] ]y + [x \gt q+\epsilon]$
\end{mex}
\section{FTC for sums}\label{S05}
If we consider the FTC, for a known function $f(x)$ it
 was found that the integral over the interval
 $[a,b]$ was the difference of two integrals
 at the end points.
\bigskip
\begin{defy}\label{DEF003}
Integration at a point 
 \[ F(a) = \int^{a} f(x)\,dx \]
\end{defy}

 Hence, a view that $\int f(x)\,dx$ is a function.
 Then we can consider the FTC as a difference of functions,
 the integrals at a point.
 FTC Theorem \ref{P018} from a point's perspective.
\[ \int_{a}^{b}f(x)\,dx = \int^{b}f(x)\,dx - \int^{a}f(x)\,dx \] 
 From a functional perspective, the integrand arguments are the
 variable parameters. That we can 
 separate the parameters, the difference, allows
 the integral to be effectively integrated at the points
 and be considered separately.
\[ F(a,b) = F(b) - F(a) \]
We can similarly consider a sum at a point.
 For example $1+2+3 + \ldots + n$ at point $n$
 is equal to $G(n)$,
 $G(n) = \frac{n(n+1)}{2}$. 
 A sum at a point is described by a function.
\bigskip
\begin{defy}\label{DEF001}
Summation at a point 
 \[ G(a) = \sum^{a} g_{j} \]
\end{defy}
 Summation and integration at a point could both
 be considered as having a fixed starting point, and
 then the integration/summation at a point
 could be 
 a unique function.
 \[ G(a,b) = G(b) - G(a) \]
 We introduce notation for sums similar to
 notation for integrals. 
\bigskip
\begin{defy}\label{DEF002}
Given $\sum_{a}^{b} g_{k}$ then 
 $\sum^{b} g_{j}$ and  
 $\sum_{a} g_{j} = -\sum^{a} g_{j}$
\end{defy}
\bigskip
\begin{theo}\label{P004} 
 A sum representation of the fundamental
 theorem of calculus. \\
 $a, b \in \mathbb{J}$ or $\mathbb{J}_{\infty}$;
\[ \sum_{k=a}^{b} g_{k} = \sum^{b} g_{k} - \sum^{a} g_{k} \]
\end{theo}
\begin{proof}
 Transfer the sum to $*G$, by Theorem \ref{P002} construct a continuous function,
 $\sum_{k=a}^{b} g_{k} = \int_{a}^{b+1} f(x)\,dx$.
 Assume the fundamental theorem of calculus is true in $*G$. 
 $\int_{a}^{b+1} f(x)\,dx$
 $= \int^{b+1} f(x)\,dx - \int^{a} f(x)\,dx$.
 Since both the sum and integral at a point
 are defined generally as functions,
 $\sum^{b} g_{j} = \int^{b+1} f(x)\,dx$,
 $\sum^{a} g_{j} = \int^{a} f(x)\,dx$.
\end{proof}

{\em RMIT University, GPO Box 2467V, Melbourne, Victoria 3001, Australia}\\
{\em chelton.evans@rmit.edu.au}
\end{document}